\documentclass[11pt]{amsart}
\usepackage{amsmath,amsfonts,amsthm,amscd,amssymb,amssymb,bm}
\usepackage{mathrsfs}
\usepackage{graphicx}
\newtheorem{theorem}{Theorem}

\newtheorem{proposition}[theorem]{Proposition}
\newtheorem{lemma}[theorem]{Lemma}

\newtheorem{corollary}[theorem]{Corollary}

\newcommand{\aaa}{\alpha}

\newcommand{\lmd}{\lambda}

\newcommand{\CP}{\mathbb{CP}}
\newcommand{\CC}{\mathbb{C}}

\newcommand{\ZZ}{\mathbb{Z}}
\newcommand{\QQ}{\mathbb{Q}}

\newcommand{\vphi}{\varphi}

\newcommand{\ol}{\overline}
\newcommand{\lra}{\longrightarrow}
\newcommand{\lras}{\,\longrightarrow\,}

\newcommand{\proofend}{\hfill$\square$}

\newcommand{\inv}{^{-1}}

\newcommand{\Bs}{{\rm{Bs}}}
\newcommand{\Pic}{{\rm{Pic}}}

\newcommand{\ms}{\mathcal}

\newcommand{\vsp}{\vspace{3mm}}

\DeclareMathOperator{\mult}{mult}
\numberwithin{equation}{section}
\numberwithin{theorem}{section}

\begin{document}
\bibliographystyle{alpha} 
\title[]
{Algebraic dimension of twistor spaces
whose fundamental system is a pencil}

\author{Nobuhiro Honda}
%\address{Mathematical Institute, Tohoku University,
%Sendai, Miyagi, Japan}
%\address{Current affiliation: Department of Mathematics, Tokyo Institute of Technology, Tokyo, Japan}
%\email{honda@math.titech.ac.jp}

\author{Bernd Kreu\ss ler}
%\address{Department of Mathematics and Computer Studies,
%Mary Immaculate College,
%Limerick, Ireland}
%\email{bernd.kreussler@mic.ul.ie}

\thanks{The first named author has been partially supported by JSPS KAKENHI Grant Number 24540061.
\\
{\it{Mathematics Subject Classification}} (2010) 53A30, 53C28}
\begin{abstract}
We show that 
the algebraic dimension of a twistor space
over $n\CP^2$
cannot be two if $n>4$ and  the fundamental system
(i.e.\,the linear system associated to the half-anti-canonical bundle,
which is available on any twistor space) is a pencil.
This means that if the algebraic dimension
of a twistor space on $n\CP^2$, $n>4$, is two, then the fundamental system
either is  empty or consists of a single member.
The existence problem for a twistor space on $n\CP^2$ with algebraic dimension
two is open for $n>4$. 
\end{abstract}
\maketitle
%\setcounter{tocdepth}{1}
%\vspace{-6mm}

%%%%%%%%%%%%%%%%%%%%%%%%%%%%%%%%%%%%%%%%%%%%%%%%
\section{Introduction}
If $Z$ is a compact complex manifold,
the algebraic dimension of $Z$, usually denoted by $a(Z)$,
is defined to be the complex dimension of a projective algebraic
variety, the rational function field of which is isomorphic to that of $Z$.
We always have $a(Z)\le \dim_{\CC}Z$,
and in case of equality, $Z$ is called Moishezon.
Any projective algebraic manifold is Moishezon.
In the other extreme case, $a(Z)=0$, $Z$ has no non-constant rational
function. 

Not so many compact differential manifolds 
admit complex structures whose algebraic dimension 
ranges from zero to half of its real dimension.
Complex tori of dimension $d\ge 2$ are examples of such manifolds.
In complex dimension two, only complex tori and K3 surfaces are
examples for which all three possible values of the algebraic dimension
actually occur.
In dimension three, {\em twistor spaces} associated to 
self-dual metrics on 4-manifolds \cite{AHS} are good candidates
of such manifolds.

By a result of Campana \cite{C91}, if $Z$ is 
a Moishezon twistor space, then the base 4-manifold
is homeomorphic to $n\CP^2$, the connected sum of $n\ge 0$ copies of complex
projective planes; by convention $0\,\CP^2=S^4$.
If $n < 4$, any twistor space on $n\CP^2$ is Moishezon
as long as the corresponding self-dual metric has
positive scalar curvature \cite{KK92,P92}.
No example seems to be known of a self-dual metric of
non-positive scalar curvature on these manifolds.
Thus if $n < 4$, all known twistor spaces are Moishezon.

The situation is very different for $n\ge 4$.
First, if $n=4$, for any $a\in\{1,2,3\}$,
there actually exists a twistor space $Z$ over $4\CP^2$
which satisfies $a(Z)=a$, see \cite{P92, CK99, Hon00, Hon06}.
Also, it is known that $a(Z)\neq 0$ as long as the self-dual metric 
on $4\CP^2$ is of positive scalar curvature \cite{P92}.
Moreover, no example seems to be known of a
self-dual metric on $4\CP^2$  of non-positive scalar curvature.
Thus our understanding of possible values of the algebraic dimension is quite
satisfactory in case $n=4$, too. 

The main focus of this article is on the case $n>4$.
For any $n>4$ and $a\in\{0,1,3\}$ it is known that 
there exists a twistor space $Z$ on $n\CP^2$ which satisfies $a(Z)=a$,
\cite{DF89, P92, LB91}.
Moreover, in \cite[Main Theorem]{Fu04}, it is stated that 
there exists a twistor space $Z$ on $n\CP^2$ with $a(Z)=2$
for any $n>4$. 
These spaces satisfy $\dim|K^{-1/2}|=1$,
where $K^{-1/2}$ is the natural square root of the 
anti-canonical line bundle on $Z$,
which is available on any twistor space \cite{Hi81}.
This is in contrast to the following theorem, which is main result of this
article. 

\begin{theorem}\label{thm}
If $n>4$ and $Z$ is a twistor space on $n\CP^2$ such that 
$\dim |K^{-1/2}|=1$, then  $a(Z)\neq 2$.
\end{theorem}

At the end of this article we give a detailed explanation of the contradiction
between this result and the result of \cite{Fu04}.

By a result of \cite{Kr99},
if a twistor space $Z$ over $n\CP^2$, $n>4$, satisfies
$\dim |K^{-1/2}|>1$, then $Z$ is Moishezon.
Therefore from Theorem \ref{thm}, we obtain
\begin{proposition}
If a twistor space $Z$  over $n\CP^2$, $n>4$,
satisfies $a(Z)=2$, then the system $|K^{-1/2}|$ consists of a single
member, or is empty.
\end{proposition}
 
To the best of the authors' knowledge, existence of this kind of twistor
spaces is not known. 
Thus the existence of a twistor space on $n\CP^2$ with algebraic dimension two
seems to be not known. 
 
Here is an outline of our proof of Theorem \ref{thm}.
Let $Z$ be a twistor space over $n\CP^2$ which 
satisfies $\dim |K^{-1/2}|=1$ and assume 
that the base locus of the pencil $|K^{-1/2}|$ 
constitutes a cycle of smooth rational curves.
This assumption is always satisfied
if $n>4$ and $a(Z)=2$ (Proposition \ref{prop:cycle2}).
By blowing up this base curve and then taking a simultaneous small resolution
of all ordinary double points that appear by the blow-up, we get the  diagram
\begin{align}
\begin{CD}
Z_1 @>{\mu_1}>> Z\\
@V{f_1}VV\\
\CP^1\;.
\end{CD}
\end{align}
Here, $\mu_1$ is the composition of the 
blow-up and the small resolution, and
$f_1$ is a surjective morphism induced by
the pencil $|\mu_1^*K^{-1/2}-E|$,
where $E$ is the exceptional divisor of 
the birational morphism $\mu_1$.
Smooth fibres of $f_1$ are naturally identified with 
members of the pencil $|K^{-1/2}|$,
and they are rational surfaces.

From this fibration, if $S$ denotes a {\em generic} member of the pencil,
we have the equality 
$
a(Z)= 1 + \kappa\inv(S),
$
where $\kappa\inv(S)$ is the anti-Kodaira dimension
\cite[Corollary 4.3]{Kr99}.
Hence if $a(Z)=2$, generic members of the pencil
satisfy $\kappa\inv(S)=1$.
Moreover, for any smooth member $S$ of $|K^{-1/2}|$, 
we have $a(Z)\le 1+ \kappa\inv(S)$.
Hence if $a(Z)=2$, we have $\kappa\inv(S)\ge 1$ for 
any smooth member $S$ of the pencil.
In Proposition \ref{prop:ak2} we show that 
if some member of the pencil satisfies 
$\kappa\inv(S)=2$,
then $a(Z)=3$.
Hence if $a(Z)=2$, any smooth member $S$ of the pencil
must satisfy $\kappa\inv(S)=1$.
We shall show by contradiction that this situation can never happen.
This implies that even if {\em some} member $S$ of the pencil
satisfies $\kappa\inv(S)=1$, 
{\em generic} members of the pencil necessarily satisfy
$\kappa\inv=0$, which then implies that $a(Z)=1$.
The main tool of our analysis is the Zariski decomposition
for a divisor on a surface \cite{Z62}.%
Finally, a remark about notation.
If $|L|$ is a complete linear system on some compact complex manifold and
if we write 
$$
|L| = |L'| + D
$$
for some effective divisor $D$,
then $D$ is a fixed component of $|L|$.
But this does not imply that the system $|L'|$ is without fixed component.

We would like to express our sincere gratitude to Professor Fujiki
for his generous and helpful comments about the anti-Kodaira
dimension of general members of the pencil $|K^{-1/2}|$ on the twistor space.
These are reflected not only in the explanation in Section 4 but 
throughout this article.
We would also like to thank the referee for forcing us to improve the clarity of this article.

\section{Zariski decomposition of 
an anti-canonical divisor on a rational surface }
\label{s:Z}
In this section we fist recall basic properties of 
the Zariski decomposition of a divisor on a projective surface,
and then investigate the Zariski decomposition
of an anti-canonical divisor on a non-singular rational surface
of anti-Kodaira dimension zero or one.

Let $S$ be a non-singular projective surface,
and $C$ an effective divisor on $S$.
Then the {\em Zariski decomposition} of $C$ is the decomposition 
$$
C = P + N,
$$ 
where $P$ and $N$ are  effective $\QQ$-divisors 
(i.e.\ divisors with positive rational coefficients) or 
the zero divisor, which satisfy
\begin{enumerate}
\item[(i)] $P$ is nef, that is $P D\ge 0$ for any curve
$D$ on $S$,
\item[(ii)] if $N\neq 0$, $N$ is negative definite in the following sense:
if $N = \sum \aaa_iE_i$ with distinct, irreducible $E_i$ and 
$\alpha_{i}\ne 0$, the total intersection matrix $(E_iE_j)$ is negative
definite, 
\item[(iii)] if $N\neq 0$, $PE_i = 0$ for all $i$.
\end{enumerate}
Any effective divisor $C$ admits a unique Zariski decomposition, see
\cite{Z62,Sa84}.
$P$ is called the {\em nef part} of $C$.
Note that if  $C=P+N$ is the Zariski decomposition of $C$,
then for any  integer $m>0$,  $mC = mP + mN$  
is the Zariski decomposition of $mC$.
One of the most important properties of the Zariski decomposition 
is that it disposes of sub-divisors of $mC$ which do not contribute
to the dimension $h^0(\ms O_S(mC))$ in the following sense:

\begin{proposition}\label{prop:Z1}
(\cite[Lemma 2.4]{Sa84})
Let $S, C$ and $C= P+N$ be as above,
and $m>0$ an integer.
Then we have 
$$
|mC| = \big|\lfloor mP\rfloor\big| + \lceil mN \rceil,
$$
where for an effective $\QQ$-divisor $D$, 
$\lfloor D \rfloor$ and $\lceil D\rceil$ denote
the (integral) round down and the round up of $D$ respectively.
In particular, the divisor $ \lceil mN \rceil$
is a fixed component of $|mC|$.
\end{proposition}

Next let $S$ be a non-singular rational surface.
For an integer $m>0$, the linear system 
$|mK\inv|$ on $S$ is called the $m$-th anti-canonical system.
If this is non-empty, the associated rational 
map $\phi_m:S\to\CP^r$, $r=h^0(mK\inv)-1$, is called
the $m$-th anti-canonical map.
The anti-Kodaira dimension, $\kappa\inv(S)$, of a rational surface $S$ is
defined as  
$$
\kappa\inv(S):= 
\max_{m > 0} \;\dim \phi_m(S)
\in\{2,1,0,-\infty\}.
$$
Here, $\kappa\inv(S)=-\infty$ occurs
when $|mK\inv|=\emptyset$ for all $m>0$,
but we will not encounter this case in the following.
Typical examples of rational surfaces with $\kappa\inv=2$ are 
Del Pezzo surfaces, while
simple examples  with $\kappa\inv=1$ are 
obtained from $\CP^2$ by blowing up the 9 intersection points
of two anti-canonical (i.e.\ cubic) curves.

The anti-Kodaira dimension of rational surfaces was 
investigated in more detail by F.\ Sakai \cite{Sa84}.
His main tool was the Zariski decomposition of the anti-canonical divisor.
Even if $\kappa\inv(S)\ge 0$, the anti-canonical divisor might not be
effective. However, Sakai \cite[Lemma 3.1]{Sa84} has shown that the divisor
$K_{S}\inv$ is pseudo effective iff $\kappa\inv(S)\ge 0$.  
A divisor $C$ on a surface $S$ is called pseudo-effective if 
$CH\ge 0$ for any ample divisor $H$. Fujita \cite{Fuj79} has shown that
any pseudo effective divisor $C$ has a unique Zariski decomposition $C=P+N$
with the same properties as above except that $P$ is no longer required to be
effective. Therefore, if $\kappa\inv(S)\ge 0$ and  
$K_{S}\inv = P+N$ is the Zariski decomposition of an anti-canonical divisor,
one can define 
{\em the degree} of a surface $S$
to be the self-intersection number
\begin{align}\label{degree}
d(S) := P^2,
\end{align}
which is clearly a non-negative rational number.
Rational surfaces with $\kappa\inv = 2$ may be
characterised in terms of the degree as follows.

\begin{proposition}\label{prop:degree}
(Sakai \cite[Proposition 1]{Sa83})
%(Sakai \cite[p.\,396]{Sa84})
Let $S$ be a non-singular rational surface with $\kappa\inv(S)\ge 0$,
and let $d(S)$ be the degree of $S$ as above.
Then $\kappa\inv(S)=2$ if and only if $d(S)>0$.
\end{proposition}

In particular we have $d(S)=0$ if (and only if) $\kappa\inv(S)\in \{0,1\}$.
In the rest of this section
we focus on rational surfaces with $d(S)=0$ which admit a special
type of anti-canonical divisor. 
For this purpose, by a {\em cycle of rational curves},  we mean
either a rational curve with one node, or
 a connected, reduced, normal crossing divisor 
\begin{align}\label{cycle0}
C = C_1 + \dots + C_k
\end{align}
on $S$,
with all $C_i$ being  non-singular rational curves,
such that the dual graph of $C$ is a circle.
For simplicity, in the expression 
\eqref{cycle0}, we allow $k$ to be one, in which case $C=C_1$ means 
a rational curve with one node.
We always assume that if $k\ge 2$ in \eqref{cycle0}
the components $C_i$ and $C_{i+1}$ intersect and
$C_{k+1} = C_1$.
By an {\em anti-canonical cycle} we mean a cycle of rational curves
which belongs to the anti-canonical class.

The following well-known lemma will be used frequently.
\begin{lemma}\label{lem:negdef}(\cite[p.\ 28, Lemma]{BoMu})
  Let $S$ be a smooth surface, $C_{1},\ldots,C_{k}$ irreducible curves on $S$
  for which $\sum C_{i}$ is a connected curve, $p_{i}>0$ rational numbers so
  that $P=\sum_{i=1}^{k} p_{i}C_{i}$ satisfies $PC_{j}=0$ for $j=1,\ldots,k$.

  Then, the intersection matrix $\left(C_{i}C_{j}\right)_{1\le i,j\le k}$ is
  negative semi-definite with one-dimensional kernel generated by $P$. More
  precisely, this means that, for $D=\sum_{i=1}^{k} r_{i}C_{i}$ with
  $r_{i}\in \mathbb{Q}$ we always have $D^{2}\le 0$ and $D^{2}=0$ occurs iff
  $D=rP$ for some $r\in\mathbb{Q}$.
\end{lemma}

For the rest of this section, let $S$ be a non-singular rational surface and
suppose that $S$ has an anti-canonical cycle $C$ as in \eqref{cycle0} and let
$C=P+N$ be its Zariski decomposition. 
Let $m_0>0$ be the smallest positive integer for which $m_0P$ is integral, 
and write  
\begin{equation}\label{eq:mzero}
  m_0P = l_1 C_1 + \dots + l_k C_k,\quad (l_i\in\ZZ_{\ge 0}).
\end{equation}

\begin{lemma}\label{lem:Pzero}
  \begin{enumerate}
  \item[\em (i)] If $P=0$ then $\kappa\inv(S)=0$.
  \item[\em (ii)] $P^{2}=0$ if and only if $P C_i=0$ for $1\le i\le k$.
  \end{enumerate}
\end{lemma}

\proof
(i) If $P=0$, the nef part of the Zariski
decomposition for $mC$ is $0$ for any $m>0$. By Proposition \ref{prop:Z1}
this means that $h^0(mK\inv) = 1$ for all $m>0$, hence $\kappa\inv(S)=0$. 

(ii) From \eqref{eq:mzero} it is clear that $P C_i=0$ for $1\le i\le k$
implies $P^{2}=0$. 

For the converse, if an index $i$ satisfies $l_i<m_0$ in \eqref{eq:mzero}, 
then $N$ includes $C_i$, and so $P C_i=0$ from the property (iii)
of the Zariski decomposition. Therefore we have
\begin{align}\label{Z5}
  (m_0P)^2 &= m_0P \Big(\sum_{1\le i\le k}l_iC_i\Big)
  = m_0 \Big(\sum_{1\le i\le k}l_iP  C_i\Big)
  = m_0^2\sum_{l_i=m_0}P  C_i.
\end{align}
Further, as $P$ is nef, we have $P C_i\ge 0$ for any $i$.
Because we assume $P^2 =0$, \eqref{Z5} implies that $PC_i=0$
even when $l_i=m_0$.
\proofend

\vsp
By $\Pic^0(C)$ we denote the group of line bundles on the anti-canonical
cycle $C$ that are trivial on each component $C_{i}$.
Thus, part (ii) of Lemma \ref{lem:Pzero} says that $P^{2}=0$ is equivalent to
$m_0P|_C\in\Pic^0(C)$.

\begin{lemma}\label{lem:blow-down}
  Let $\pi:S\to \ol S$ be the blow-down of a $(-1)$-curve on $S$, and let
  $\ol C:=\pi(C)$ be the image of the anti-canonical
  cycle $C$, and $x\in \ol S$ the image of 
  the $(-1)$-curve.
  Define $m=\mult_{x}\ol C$. Then,
  \begin{enumerate}
  \item [\em (i)] $\ol C$ is an anti-canonical cycle on $\ol S$ and
    $m\in\{1,2\}$;
  \item [\em (ii)] if $\ol C=\ol N + \ol P$ is the Zariski
    decomposition of $\ol C$ and $m=2$, then $P=0$ iff $\ol P=0$.
  \end{enumerate}
  
\end{lemma}

\begin{proof}
  (i)
  Let $E\subset S$ be the exceptional curve of the blow-up $\pi$. As
  $EC=EK\inv=1$, either $E$ is a component of $C$ or it intersects $C$
  transversally at one point. Let $C'$ be the strict transform of $\ol
  C=\pi(C)$, then $C'=C$ if $E$ is not a component of $C$, or $C'=C-E$ if $E$
  is a component of $C$. In both cases, $\ol C$ clearly is a cycle of rational
  curves. 
  Because $m=C'E$ and $EC=1$, we obtain $m=1$ if $E$ is not a component
  of $C$, and $m=2$ if $E$ is a component of $C$. In particular,
  $C=C'+(m-1)E$ in both cases. Because $\pi^{\ast}\ol C=C'+mE$ and
  $K\inv=\pi^{\ast}K_{\ol S}\inv -E$, we now obtain $\ol C\in|K_{\ol S}\inv|$.

  (ii) As $m=2$, the exceptional curve $E$ is a component of $C$. We fix
  notation so that $C_{k}=E$ and $\ol C=\ol C_{1}+\cdots+\ol C_{k-1}$, where
  $\ol C_{i}=\pi(C_{i})$. Consider the intersection matrix
  $\ol M=\left(\ol C_{i}\ol C_{j}\right)_{1\le i,j\le k-1}$.
  The intersection matrix
  $M=\left(C_{i}C_{j}\right)_{1\le i,j\le k}$ is obtained from $\ol M$ by adding
  an extra row and a column with $C_{1}C_{k}=C_{k-1}C_{k}=1$ and $C_{k}^{2}=-1$
  being their only non-zero entries. In addition, four entries in $\ol M$ have
  to be changed to get $M$, namely $C_{1}^{2}=\ol C_{1}^{2}-1$,
  $C_{k-1}^{2}=\ol C_{k-1}^{2}-1$ and $C_{1}C_{k-1}=0$, whereas
  $\ol C_{1}\ol C_{k-1}=1$.

  By adding the $k$-th column of $M$ to columns $1$
  and $k-1$, the values of the entries of the part of $M$ that corresponds to
  $\ol M$ are restored to the original values they had in $\ol M$. It now
  follows easily from Sylvester's criterion that the matrix $M$ is negative
  definite if and only if $\ol M$ is so.
  Because $\ol P=0$ is equivalent to $\ol M$ being negative definite and
  $P=0$ if and only if $M$ is negative definite, the claim now follows. 
\end{proof}

We will frequently need the following more detailed properties of the Zariski
decomposition $C=P+N$ of the anti-canonical cycle $C$ on the surface $S$.

\begin{proposition}\label{prop:Z2}
  Suppose that $P\ne 0$ and $P^{2}=0$.
\begin{enumerate}
\item[\em (i)] 
  We have $l_i>0$ for all $i$.
\item[\em (ii)]
  If $\pi:\widetilde{S}\to S$ is the blow-up of a point $x\in C$ and
  $m=\mult_{x}C$, then
  \begin{align*}
    \kappa\inv(\widetilde{S}) &= \kappa\inv(S)&\text{ if } m=2\phantom{.}\\
    \kappa\inv(\widetilde{S}) &=0             &\text{ if } m=1.
  \end{align*}
\item[\em (iii)]
  We have $l_i=1$ for some index $i$.
\item[\em (iv)]
  If $K^2<0$, then $m_0>1$, $k\ge 2$, and 
  we have $l_i\neq l_j$ for some indices $i$ and $j$.
\end{enumerate}
\end{proposition}

\proof 
For (i),
if $l_i=0$ for some $i$,
we have $m_{0}PC_i = l_{i-1} + l_{i+1}$ since $C$ is a cycle.
But as $P^{2}=0$, Lemma \ref{lem:Pzero} (ii) shows that $PC_i=0$ for all $i$,
hence $l_{i-1} = l_{i+1}=0$.
Repeating this argument, we obtain $l_i=0$ for any $i$, 
but this contradicts $P\neq0$.

To prove (ii) we recall from \cite[Lemma 5.1]{Kr99} that
$\kappa\inv(\widetilde{S}) \le \kappa\inv(S) = 1$ with equality 
if $m\ge2$. Because $C$ is a cycle of rational curves, the points on $C$ have
multiplicity $1$ or $2$. We assume now $m=1$
and show $\kappa\inv(\widetilde S)=0$. 
From  (i), Lemma \ref{lem:Pzero} (ii) and Lemma \ref{lem:negdef},
for any divisor $D=\sum_{i=1}^{k}n_{i}C_{i}$,
we have $D^{2}\le 0$ with equality only if $D=rP$ for some $r\in\mathbb{Q}$.

Let $\widetilde{C}_{i}$ be the strict transform of $C_{i}$ and
$\widetilde{C} = \sum_{i=1}^{k}\widetilde{C}_{i}$. Because we assumed $m=1$,
$\widetilde{C}$ is an anti-canonical cycle on $\widetilde{S}$. We now show that
the intersection matrix of the components of $\widetilde{C}$ is negative
definite. To fix notation, we assume $x\in C_{1}$, so that
$\widetilde{C}_{1}^{2} = C_{1}^{2}-1$. Consider
$\widetilde{D}=\sum_{i=1}^{k} n_{i}\widetilde{C}_{i}$ and let
$D=\sum_{i=1}^{k} n_{i}C_{i}$, then $\widetilde{D}^{2} = D^{2}-n_{1}^{2}\le
D^{2} \le 0$. If $\widetilde{D}^{2} = 0$ we then have $D^{2}=0$ and $n_{1}=0$,
hence $D=rP=\sum_{i=1}^{k} rl_{i}C_{i}$ for some $r\in \QQ$ and so 
$rl_{1}=n_{1}=0$. As $l_{1}\ne 0$ this implies $r=0$, hence $D=0$, which gives
$\widetilde{D}=0$. This shows that the intersection matrix
$(\widetilde{C}_{i}\widetilde{C}_{j})_{1\le i,j\le k}$ is negative
definite. This implies that $\widetilde{P}=0$ in the Zariski decomposition
$\widetilde{C}=\widetilde{P}+\widetilde{N}$of the anti-canonical divisor
$\widetilde{C}$, 
hence $\kappa\inv(\widetilde{S})=0$ by Lemma \ref{lem:Pzero} (i). 

To prove (iii), we recall from 
\cite[Theorem 3.4]{Sa84} 
that there is a birational morphism
$\vphi:S\to S_{0}$ with $S_{0}$ being non-singular, such that 
$\kappa\inv(S_{0}) = \kappa\inv(S)$ and any anti-canonical divisor
$C_{0}\in|K_{S_{0}}\inv|$ is nef, that is $N_{0}=0$
in the Zariski decomposition $C_{0} = P_{0}+N_{0}$ of $C_{0}$. 
Because $P^{2}=0$ and $\kappa\inv(S_{0}) = \kappa\inv(S)$, 
Proposition \ref{prop:degree} implies that 
$K_{S_{0}}^2 = P_{0}^{2} = 0$ as well.

The morphism $\vphi$ is a composition of blow-ups. By Lemma
\ref{lem:blow-down} (i) the image of $C$ in each of these partial blow-ups is
an anti-canonical cycle containing the blown-up point. Moreover, as we have
seen in the proof of (ii) above, if the blown-up point had multiplicity one on
the anti-canonical cycle, the nef part of the Zariski decomposition would
vanish after the blow-up. Because of Lemma \ref{lem:blow-down} (ii), this
would lead to $P=0$ on $S$, in contradiction to our assumption. Therefore, it
follows that at each step a double point of the anti-canonical cycle is
blown-up and the nef part of the Zariski decomposition does not vanish. 

Because the anti-Kodaira dimension does not increase under blow-up,
\cite[Lemma 5.1]{Kr99}, and $\kappa\inv(S_{0}) = \kappa\inv(S)$, all
partial blow-ups in this process have $\kappa\inv=\kappa\inv(S)$, in
particular the self-intersection number of the nef part of the Zariski
decomposition is equal to zero.

From the above we infer that $C_{0}:=\vphi(C)$ is an anti-canonical cycle on
$S_{0}$ which coincides with the nef part, $P_{0}$, of its Zariski
decomposition. Thus, $m_{0}=1$ and all $l_{i}=1$ on $S_{0}$, and so the
assertion (iii) holds for the pair $(S_{0}, C_{0})$. 
To finish the proof of (iii) by induction, it remains to show that, 
if $(S, C)$ is a pair consisting of a rational surface and an
anti-canonical cycle on it that satisfies the assumptions of the current
proposition and has $l_{i}=1$ for one index $i$,
and if $\pi:\widetilde{S}\to S$ is the blow-up at a double point of $C$, then
the nef part of the new anti-canonical cycle $\widetilde{C}$ on
$\widetilde{S}$ also satisfies (iii).

To show this, we continue to use the notation introduced above, like
\eqref{cycle0}, for $S,C$ and $P$. To fix notation, we suppose that the point
$C_{k}\cap C_{1}$ is blown up by $\pi$. 
Let $\widetilde{C}_{k+1}$ be the exceptional divisor  of $\pi$ and
write $\widetilde{C}_i$ for the strict transform of the component $C_i$ ($1\le
i\le k$). 
We recall that $l_{i}>0$ ($1\le i\le k$), define
$l_{k+1}=l_{1}+l_{k}$ and consider the divisor 
$\widetilde{D}$ on $\widetilde{S}$
\begin{align}\label{ZD1}
\widetilde{D} = \sum_{i=1}^{k+1} l_i \widetilde{C}_i\,.
\end{align}
It is easy to see that $\widetilde{D}\widetilde{C}_i=0$ for $i=1,\ldots,k+1$,
and therefore $\widetilde{D}$ is nef and $\widetilde{D}^2=0$.
Hence, again by Lemma \ref{lem:negdef}, the intersection matrix
$(\widetilde{C}_i\widetilde{C}_j)_{1\le i\le k+1}$ is negative semi-definite,
and if $\widetilde{P}$ is the nef part of $\widetilde{C}
= \sum\widetilde{C}_i$ and
$\widetilde{m}_0$ is the smallest positive integer  for which
$\widetilde{m}_0\widetilde{P}$ is integral, we have  
$\widetilde{m}_0\widetilde{P} = r\widetilde{D}$ for some $r\in \QQ_{>0}$.
Since $l_i=1$ for some index $i$ ($i\le k$) by the inductive assumption,
it follows that $r\in\ZZ_{>0}$.

If $r\neq 1$, all the coefficients of $\widetilde{m}_0\widetilde{P}$ would be
divisible by $r$.
If all $\widetilde{C}_{i}$ appeared in the negative part, $\widetilde{N}$,
of the Zariski decomposition of $\widetilde{C}$, 
condition (ii) in the definition of the Zariski decomposition would imply 
$\widetilde{P}^{2}<0$. Thus, $\widetilde{P}^{2}=0$ implies that at least one of
the $\widetilde{C}_{i}$ is missing in $\widetilde{N}$.
Hence, at least one of the
coefficients of $\widetilde{m}_0\widetilde{P}$ is equal to
$\widetilde{m}_0$. But then $r$ would divide $\widetilde{m}_0$ and all
coefficients of $\widetilde{m}_0\widetilde{P}$, which contradicts the choice
of $\widetilde{m}_0$.
This shows that $r=1$, and we obtain $\widetilde{m}_0\widetilde{P} =
\widetilde{D}$. 
Thus from the expression \eqref{ZD1} we now see that the cycle $\widetilde{C}$
satisfies the property (iii). 

The final item (iv) follows now easily, because $K^{2}<0$ implies that the
morphism $\vphi$ in the proof of (iii) is not an isomorphism and so there was
at least one blow-up carried out. The expression \eqref{ZD1} for $m_0P$ shows
then that at least one of the $l_{i}$ is greater than $1$. Also, note that
$m_{0}$ is equal to the largest coefficient in $m_0P$. 
\proofend

\begin{lemma}\label{prop:Z21}
  Suppose $P\ne 0$, $P^{2}=0$ and that there exists an integer $\nu>0$ for
  which $h^{0}(\nu m_{0}P)>1$, then $m_0P|_C\in \Pic^0(C)$ has finite order.
  Moreover, if $r$ is the smallest positive integer for which
  $h^{0}(rm_{0}P)>1$, then $rm_{0}P|_{C}\simeq\ms O_{C}$.
\end{lemma}

\proof
Let $r$ be the smallest positive integer for which
$h^{0}(rm_{0}P)>1$ and let $s\in H^0(r m_0P)$ be a non-zero element that
satisfies $(s)\neq r m_0P$. 
Put $D:=(s)\in |rm_0P|$ and write
$$
D = D' + \sum_{1\le i\le k} a_iC_i,
$$
where  $D'$ (which may be $0$ at this moment)
 does not include $C_i$ for any $i$.
As $D\neq rm_0P=\sum rl_iC_i$,
we have $a_i\neq rl_i$ for at least one $i$.
Since  $D\in |rm_0P|$, we have a linear equivalence
\begin{align*}
D' +  \sum_{1\le i\le k} a_iC_i
\sim
\sum_{1\le i\le k} rl_i C_i.
\end{align*}
Collecting all indices that satisfy
$a_i>rl_i$ (if any) on the left-hand side, we obtain 
\begin{align}\label{le1}
D' + \sum_{a_i>rl_i}(a_i-r l_i)C_i
\sim \sum_{a_j \le rl_j}(r l_j - a_j)C_j.
\end{align}
If RHS is the zero-divisor,
then so has to be LHS, which contradicts
$a_i\neq rl_i$ for some $i$.
So both sides of \eqref{le1} are effective divisors.
Because LHS and RHS  do not have a common 
irreducible component,
we obtain that the self-intersection number of RHS is non-negative.
Since $P^{2}=0$, by Lemma \ref{lem:Pzero} (ii) 
we have $PC_i=0$ for any $i$, $1\le i\le k$. 
Because of Proposition \ref{prop:Z2} (i) we can apply
Lemma \ref{lem:negdef} to obtain that the total intersection matrix
$
\big(C_iC_j\big)_{1\le i,j\le k}
$
is  negative semi-definite.
Hence the self-intersection number of RHS has to be zero.
By Lemma \ref{lem:negdef}, this can happen only when
$$
 \sum_{a_j \le rl_j}(r l_j - a_j)C_j = r'm_0P,
$$
for some $r'\in\QQ_{>0}$.
Because $l_i>0$ for all $i$ by Proposition \ref{prop:Z2} (i), 
we can conclude that $rl_i-a_i>0$ for all $i$.
Therefore from \eqref{le1} we obtain a linear equivalence
\begin{align}\label{le2}
D'\sim r'm_0P, \quad
 r'\in\QQ_{>0}\,.
\end{align}
On the other hand, we have $l_i=1$ for some index $i$ by 
Proposition \ref{prop:Z2} (iii).
Because $D'$ is integral, \eqref{le2} implies now that $r'\in \ZZ$.
Hence by the minimality of $r$,
and as $D'\le D$, we have $r'=r$.

Because we have chosen $D'$ not to have any $C_{i}$ as a component and because
$D'C_i=r'm_0PC_i=0$, the divisor $D'$ does not intersect  any $C_i$.
As $D'\sim rm_0P$, this means that $rm_0P|_C\simeq \ms O_C$.
Hence $m_0P|_C$ is of finite order in $\Pic^0(C)$.
\proofend

\vsp
We will use this result to show the following property regarding
pluri-anti-canonical systems on $S$,
which will be used in the next section.

\begin{proposition}\label{prop:Z3}
If $P\ne0$, $P^{2}=0$ and $m_0P|_C\in\Pic^0(C)$ is of finite order $\tau$, 
then
\begin{enumerate}\renewcommand{\labelenumi}{\textit{(\roman{enumi})}}
\item
$ |\tau m_0K\inv|= |\tau m_0P| + \tau m_0N$,
$\,\Bs\,|\tau m_0 P|=\emptyset$,  
$\,\dim |\tau m_0P|=1$, and the associated morphism
$S\to\CP^1$ is an elliptic fibration;
\item $\kappa^{-1}(S)=1$;
\item 
for any $\nu>0$, the system $|\nu\tau m_0P|$
is composed with the pencil $|\tau m_0P|$, i.e.\ each element of
$|\nu\tau m_{0}P|$ is a sum of elements of $|\tau m_{0}P|$;
\item
for any  integer $\nu>0$, we have
$$
\big|\nu\tau m_0P-C\big| = \big|(\nu-1)\tau m_0P\big| +
\big(\tau m_0P -C\big).
$$
\end{enumerate}
\end{proposition}

\proof
Because $C$ is an anti-canonical divisor with Zariski decomposition
$C=P+N$, we obtain $|\tau m_0K\inv|= |\tau m_0P| + \tau m_0N$. 
By Lemma \ref{prop:Z21} we have
\begin{align}\label{one1}
h^0(\nu m_0P)= 1,\quad \text{ if }\;\; 0<\nu<\tau.
\end{align}
Note that $\tau m_0 P - C$ is an effective divisor by 
Proposition \ref{prop:Z2} (i). Using Serre-duality this implies
$H^2(\tau m_0P-C)=0$. 
Using $P^{2}=0$, the definition of the Zariski decomposition implies that
\[  K^{2}=C^{2}=P^{2}+2PN+N^{2}=N^{2}\le 0  \] 
with equality iff $N=0$. If $(\tau m_0P-C)^2 = C^2=K^2<0$,
the non-empty system $|\tau m_0P-C|$ has a (non-zero) fixed component.
Indeed, if $L$ is a line bundle on a smooth surface so that $|L|$ is not empty
and $L^2<0$, then $|L|$ must have a fixed component, as can be seen as
follows. 
Let $Y=\sum d_i Y_i \in |L|$ with $d_i>0$ and prime divisors $Y_i$. Then there
exists $k$ such that $LY_k<0$, because otherwise we would have $L^2 =
L\sum d_i Y_i = \sum d_i LY_i \ge 0$.
If now $Y'$ is any element of $|L|$, then $Y'Y_k = LY_k<0$, hence $Y_k$ is a
component of $Y'$, i.e.\ a fixed component of $|L|$. 

If $K^{2}<0$ we let $D$ be the (maximal) fixed component
of  $|\tau m_0P-C|$.
By Lemma \ref{lem:Pzero} (ii), Proposition \ref{prop:Z2} (i) and 
Lemma \ref{lem:negdef}, the
intersection matrix formed by the components of the cycle $C$ is negative
semi-definite, hence $(\tau m_0P-C-D)^2\le 0$.
But if this was negative, $|\tau m_0P-C-D|$ would still have
a fixed component, in contradiction to the choice of $D$.
Hence $(\tau m_0P-C-D)^2 = 0$.
If $K^2=0$, we simply take $D=0$ to obtain $(\tau m_0P-C-D)^2 = 0$. 
Therefore, by Lemma \ref{lem:negdef} again, we have
\begin{align}\label{0028}
\tau m_0P-C-D =s m_0P,\quad s\in\QQ_{>0}.
\end{align}
Moreover, $s$ is an integer since $m_0P$ has a component
of multiplicity one, by Proposition \ref{prop:Z2} (iii),
and the left-hand side is an integral divisor.
From \eqref{0028} we have $s<\tau$, and so $h^0(sm_0P)=1$ by \eqref{one1}. 
Therefore, 
\[
h^0(\tau m_0P-C) = h^0(\tau m_0P-C-D) = h^0(sm_0P)=1\,.
\]
By Riemann-Roch we readily obtain $\chi(\tau m_0P-C) = 1$. Because
$H^{2}(\tau m_0P-C)=0$ and $h^{0}(\tau m_0P-C)= 1$, we must have 
$H^1(\tau m_0P-C) = 0$.
Since $h^0(\tau m_0P|_C)= h^0(\ms O_C)=1$, the standard exact sequence 
$$
0 \lras \tau m_0P-C \lras \tau m_0P \lras \tau m_0P|_C\lras 0
$$
implies now that $h^0(\tau m_0P) = 2$. In addition, this sequence provides us
with a surjection $H^0(\tau m_0P)\to H^0(\tau m_0P|_C) \simeq H^0(\ms O_C)$,
from which we obtain $\Bs\,|\tau m_0P|=\emptyset$. 

The general fibre of the morphism $\phi:S\to \CP^1$ associated to the
pencil $|\tau m_0P|$ is non-singular, because $S$ is smooth.
The smooth (and hence all) fibres of $\phi$ are connected. To see this, let
$F$ be a smooth fibre of $\phi$, then $h^{0}(F) = h^{0}(\tau m_0P) =2$ and
$\ms O_{S}(F) \cong \phi^{\ast}\ms O(1)$, hence
$\ms O_{F}(F) \cong \phi^{\ast}\ms O(1)_{\vert F}\cong \ms O_{F}$.
Because $h^{1}(\ms O_{S}) = 0$, the exact sequence
\[0 \lras \ms O_{S} \lras F \lras \ms O_{F}(F) \lras 0\]
implies $2=h^{0}(F) = h^{0}(\ms O_{S}) + h^{0}(\ms O_{F})
= 1+h^{0}(\ms O_{F})$. Therefore, $h^{0}(\ms O_{F})=1$, i.e.\ $F$ is
connected.
Then, since $P^{2}=PK\inv=0$, the genus formula implies that the
general fibre of $\phi$ is an elliptic curve.
This completes the proof of (i).

To prove (ii) we just have to observe that (i) immediately implies 
$\kappa^{-1}(S)\ge1$. The assumption $P^{2}=0$, on the other hand, gives
$\kappa^{-1}(S)<2$, and so $\kappa^{-1}(S)=1$.

To prove assertion (iii), first observe that
$|\nu\tau m_{0}P| = |\phi^{\ast}\mathcal{O}(\nu)|$. As $\phi$ has connected
fibres, we have $\phi_{\ast}\ms O_{S} \cong \ms O_{\mathbb{P}^{1}}$ and the
projection formula implies
$\phi_{\ast}\phi^{\ast}\mathcal{O}(\nu) \cong \mathcal{O}(\nu)$. Hence,
$H^{0}(\nu\tau m_{0}P) \cong H^{0}(\phi^{\ast}\mathcal{O}(\nu))
\cong H^{0}(\mathcal{O}(\nu))$, which shows that each element of
$|\nu\tau m_{0}P|$ is a sum of fibres, i.e.\ a sum of elements of
$|\tau m_{0}P|$. 

Finally, for (iv), let $\phi:S\to\CP^1$ be as above and 
let $t_1\in\CP^1$ be the point for which $\tau m_0P= \phi^{-1}(t_1)$.
Let $D\in |\nu\tau m_0P-C|$ be any member.
Then since $D+C\in |\nu\tau m_0P|$ and each member of this linear system is a
sum of members of $|\tau m_0P|$, i.e.\ fibres of $\phi$,
and because $C\subset \phi^{-1}(t_1)$
there exist points $t_2,t_3,\dots, t_{\nu}$ such that 
$$
D + C = \sum_{1\le i\le \nu} \phi\inv(t_i)\quad\text{ and so }\quad
D =  \sum_{2\le i\le \nu} \phi\inv(t_i) + \big(\phi\inv(t_1) - C\big)\,.
$$
This implies assertion (iv), because $h^0(\tau m_0P-C) = 1$, as we have shown
above.
\proofend

\begin{corollary}\label{cor:finord}
  Suppose that $P\ne 0$ and $P^{2}=0$. Then
  \[
  \kappa\inv(S) = 1 \quad\iff \quad
  m_{0}P|_{C}\;\text{ has finite order in }\; \Pic^{0}(C)\,.
  \]
\end{corollary}
\proof
If $\kappa\inv(S)=1$ there exists an integer $\nu>0$ for which 
$h^0(\nu m_0P)=h^0(\nu m_0K\inv)>1$, hence $m_{0}P|_{C}$ 
has finite order by Lemma \ref{prop:Z21}.
The converse is Proposition \ref{prop:Z3} (ii).
\proofend

\section{Twistor spaces whose fundamental system is a pencil}

Let $Z$ be a twistor space on $n\CP^2$,
and $F$ be the natural square root of the anti-canonical bundle
over $Z$, which is known to exist on any twistor space \cite{Hi81}.
Following Poon \cite{P92}, we call $F$ the {\em fundamental 
line bundle}, and the associated linear system $|F|$
the {\em fundamental system}.
Basic properties of the fundamental line bundle are
$$
FL = 2, {\text{ and }}\, \sigma^*F\simeq \ol F,
$$
where $L$ is a fibre of the twistor projection $Z\to n\CP^2$
which is called a \emph{real twistor line},
and $\sigma:Z\to Z$ is a natural anti-holomorphic involution
called the \emph{real structure}.

By the works of Poon \cite{P86,P92}, Kreussler and Kurke
\cite{KK92},
when $n\le 3$, for any twistor space $Z$ over 
$n\CP^2$ whose self-dual metric is of positive
scalar curvature, we have $\dim |F|\ge 3$,
and the structure of these twistor spaces is well understood through
the rational map associated to the fundamental system $|F|$.
In particular, all of these twistor spaces are Moishezon.

Let now $Z$ be a twistor space on $n\CP^2$,
$n\ge 4$, and assume that $\dim |F|\ge 2$.
In this situation $|F|$ needs to contain an irreducible divisor, because
otherwise there would exist two pencils, $|D|$ and $|\overline{D}|$, of
divisors of degree one that define a surjective rational map to
$\CP^{1}\times\CP^{1}$ and for which $D+\overline{D}$ is a fundamental
divisor. But then the image of $|F|$ would be at least two-dimensional and
Bertini's Theorem implies that not all elements in $|F|$ could be
reducible. This allows us to apply the results of \cite{Kr99}.

First, by \cite[Theorem 3.6]{Kr99}, if $\dim |F|\ge 3$, we always have the
equality $\dim |F| = 3$, and $Z$ has to be a so-called LeBrun twistor space
\cite{LB91}, the structure of which is also well-understood.
In particular, such twistor spaces are Moishezon.
Second, by \cite[Theorem 3.7]{Kr99}, if $\dim |F|=2$ and $n\ge 5$,
then $Z$ has to be one of the 
twistor spaces investigated in \cite{CK98},
and they are again Moishezon.
Third, if $\dim |F|=2$ and $n=4$, then $Z$ is either one of the 
Moishezon twistor spaces studied in \cite{CK98},
or a twistor space that satisfies $a(Z)=2$.
The former happens  exactly when $\Bs|F|\neq\emptyset$
(\cite[Proposition 2.4]{Hon06}),
and
if the latter is the case, the morphism $Z\to\CP^2$ associated to 
the net $|F|$ is an elliptic fibration which is
an algebraic reduction of $Z$.
Thus the basic structure of $Z$ is also well-understood if $\dim |F|=2$.

For the rest of this paper we let $Z$ be a twistor space over
$n\CP^2$ $(n\ge 4)$ and suppose  $\dim|F| = 1$.
Then general members of the fundamental system $|F|$ are irreducible,
since otherwise we readily have $\dim |F|\ge 3$. This implies that the
self-dual metric on $n\CP^2$ has positive scalar curvature, 
see \cite[Proposition 2.4]{Kr99}.

Let $S\in |F|$ be a smooth fundamental divisor and
recall that $H^1(\ms O_Z)=0$ because $Z$ is simply connected.
As we assume $h^0(F)=2$, the standard exact sequence
\begin{align}\label{stseq}
  0 \lras \ms O_Z\lras F\lras K_S\inv\lras 0,
\end{align}
implies $h^0(K_S\inv)=1$.
This means that the anti-canonical system $|K_S\inv|$ consists
of a single member, say $C$. 
In particular, we have $\kappa\inv(S)\ge 0$.
From the surjectivity of the restriction map $H^0(F)\to H^0(K_S\inv)$ 
we have
\begin{align}
  \Bs\,|F| = C.
\end{align}

By a theorem of Pedersen and Poon
\cite{PP94}, any real irreducible member $S\in |F|$
is non-singular and 
obtained from $\CP^1\times\CP^1$ by blowing up $2n$ points.
On such a surface, the anti-canonical curve $C$ is  real (i.e.\ $\sigma(C)=C$)
since $S$ and so $K_S\inv$ are real. 
Moreover we have the following  result on the structure of the base curve $C$.

\begin{proposition}\label{prop:cycle}
  Let $Z\to n\CP^2$, $n\ge 4$, be a twistor space satisfying $\dim |F|=1$, and 
  let the curve $C$ be the base locus of the pencil $|F|$, as above.
  If $C$ is non-singular, it is an elliptic curve.
  If $C$ is singular, it is a cycle of rational curves
  which is of the form
  \begin{align}\label{cycle1}
    C = C_1 + \dots + C_k + \ol C_1 + \dots + \ol C_k
  \end{align}
  for some $k\ge 1$,
  where $\ol C_i$ means $\sigma(C_i)$, and 
  in the presentation \eqref{cycle1} two components
  intersect iff they are adjacent, or they are $C_1$ and $\ol C_k$.
\end{proposition}

\proof
Recall that $C\subset S$ for each smooth real $S\in |F|$.
If $C$ is non-singular, the adjunction formula 
immediately implies that $C$ is an elliptic curve.
If $C$ is singular it is a cycle of rational curves,
by \cite[Proposition 3.6]{Kr98}. If $C$ arises from
item (I) or (III) in \cite[Proposition 3.6]{Kr98}, it consists of conjugate
pairs of rational curves, as required in \eqref{cycle1}. If $C$ arose from
item (II) in \cite[Proposition 3.6]{Kr98} it would have exactly
two irreducible components, one of them a real twistor line. Because a real
twistor line in $S$ generates a pencil, we would readily have
$h^0(K_S\inv)=2$, in contradiction to our observations just after the exact 
sequence \eqref{stseq}. 
Therefore $C$ is a cycle of rational curves
as described in \eqref{cycle1}.
\proofend

\vsp
In the sequel, for simplicity of notation, we often write
$C_{k+1}$ for $\ol C_1$ and $\ol C_{k+1}$ for $C_1$
for components of the cycle \eqref{cycle1}.
This cycle will be significant throughout our proof of 
the main result.

Another important property of twistor spaces 
satisfying $\dim |F| =1 $,
which we next explain, concerns reducible members of the pencil $|F|$. 
The cycle \eqref{cycle1} can be split into connected halves in exactly 
$k$ ways.
For example, if $k=3$, the possibilities are:
\begin{align*}
  (C_1+C_2+C_3) & +  (\ol C_1 + \ol C_2 + \ol C_3),\\
  (C_2+C_3+\ol C_1) & +  (\ol C_2 + \ol C_3 + C_1),\\
  (C_3+\ol C_1+\ol C_2) & +  (\ol C_3 +  C_1 +  C_2).
\end{align*}
The following proposition,
which was proved in \cite{Kr98},
implies that if the base curve is singular,
these subdivisions
are nicely realised by reducible members of the pencil $|F|$.
\begin{proposition}\label{prop:d1}
  (\cite[Proposition 3.7]{Kr98})
  If the base curve of the pencil $|F|$ is a cycle
  of rational curves as in  \eqref{cycle1},
  then $|F|$ has exactly $k$ reducible members.
  Moreover each of them is real and of the form
  $S_i^++S_i^-$ ($1\le i\le k$),
  where $S_i^+$ and $S_i^-$ are non-singular irreducible divisors
  satisfying $\ol S_i^+ = S_i^-$.
  Furthermore, the divisor $S_i^++S_i^-$
  splits the cycle $C$ into  halves in the following manner:
  \begin{itemize}
  \item if $L_i$ denotes the real twistor line joining 
    the two points $C_i\cap C_{i+1}$ and $\ol C_i\cap \ol C_{i+1}$,
    then $S_i^+\cap S_i^-=L_i$,
  \item the intersections $S_i^+\cap C$ and $S_i^-\cap C$ are 
    connected.
  \end{itemize}
\end{proposition}

Note that all the reducible fundamental divisors $S_i^++S_i^-$ are singular
along the twistor line $L_{i}=S_i^+\cap S_i^-$. Therefore, all smooth
fundamental divisors are automatically irreducible.
We will also need the following property of the Zariski decomposition
of the cycle $C$.

\begin{proposition}\label{prop:indep}
  In the situation of Proposition \ref{prop:cycle},
  let $S$ be a smooth member of the pencil $|F|$.
  Then, the degree of $S$ (see \eqref{degree}) and 
  the Zariski decomposition of the cycle $C\subset S$
  are independent of the choice of $S$.
\end{proposition}

\proof
Let $S\in |F|$ be any smooth member, which is irreducible as we have seen
above, and let $C\subset S$ the base locus of the
pencil $|F|$. For the self-intersection numbers in $S$ of components of
the cycle $C$ we have 
$$
C_i^2 = -2 + K_S\inv C_i = -2+FC_i\,,
$$
hence these self-intersection numbers in $S$ are 
independent of the choice of the smooth member $S$.
Obviously the intersection numbers between different components
of $C$ are independent of the choice of $S$ as well.

Let $C=P+N$ be the Zariski decomposition of
the cycle $C$ regarded as a curve in $S$.
Then $P$ is nef as a divisor in $S$. 
In particular $PC_i\ge 0$ for any index $i$.
As the intersection number $PC_i$ is determined by the coefficients of $P$ and
the intersection numbers $C_{j}C_{k}$, it is independent of the choice
of $S$ and it follows that $P$ is nef also in any other smooth member of
$|F|$. 
Similarly, $N$ is negative definite not only in $S$ but also 
in any other smooth member of $|F|$.
By the same reason, we have $PC_i=0$
for each $C_i$ which is included in $N$, not only in $S$ but 
also in any other smooth member of $|F|$.
Thus all the properties (i), (ii) and (ii) that characterise
the Zariski decomposition are satisfied for $P+N$ in all smooth
members of the pencil $|F|$.
Therefore, the Zariski decomposition
of the cycle $C$ is independent of the choice
of a smooth member $S\in |F|$.
This implies that the self-intersection number $P^2=d(S)$
is also independent of the choice of $S$.
Hence we obtain the proposition.
\proofend

\vsp
Note that, even though the Zariski decomposition $C=P+N$ is independent of
$S$, it is possible that the divisor $m_{0}P|_{C}$ in $\Pic(C)$ does depend
on $S$. The reason is that the restriction $m_{0}P|_{C}$ actually is
$\ms O_{S}(m_{0}P)\otimes\ms O_{C}$ and like
$\ms O_{S}(C)\otimes\ms O_{C}$, the normal bundle of $C$ in $S$, this may
depend on $S$. The independence of the Zariski decomposition just means that
the rational coefficients of $P$ and $N$ do not depend on $S$.

\vsp
Since the works of Poon \cite{P88}
and Campana \cite{C91-2}, it has long been well realised that 
the algebraic dimension $a(Z)$ of a twistor space $Z$
on $n\CP^2$ is closely related to 
the anti-Kodaira dimension of a smooth member of
the system $|F|$.
Actually, if $S$ is a smooth
member of $|F|$, the standard exact sequence
\begin{align}
0 \lras (m-1)F \lras mF\lras mK_S\inv\lras 0,\quad m>0
\end{align}
and the equality $a(Z)= \kappa(Z,F)$ 
imply the inequality
\begin{align}\label{ineq1}
a(Z)\le 1 + \kappa\inv(S).
\end{align}
Does equality hold in \eqref{ineq1} when $Z$ satisfies $\dim |F|=1$?
For instance, if at least one smooth member
$S$ of the pencil $|F|$ satisfies 
$\kappa\inv(S)=0$, we clearly have  the equality $a(Z)=1$.
The following proposition shows that the same is true if some $S\in|F|$
satisfies $\kappa\inv(S)=2$.

\begin{proposition}\label{prop:ak2}
If at least one smooth member $S$ of the pencil $|F|$ satisfies
$\kappa\inv(S)=2$, then equality holds in \eqref{ineq1},
and the twistor space $Z$ is Moishezon.
\end{proposition}

\proof
Let $S\in|F|$ be smooth with $\kappa\inv(S)=2$.
If the base curve $C$ of  the pencil $|F|$ is non-singular, we readily obtain
$\kappa\inv(S)\le 1$ from $C^2 = K_S^2 = F^3 =8-2n\le 0$.
Hence $C$ is singular, and by Proposition \ref{prop:cycle},
$C$ is a cycle of rational curves on $S$ as in \eqref{cycle1}.
Let $C=P+N$ be the Zariski decomposition of $C$ on $S$.
For the degree of $S$ we then have $d(S)>0$ by 
Proposition \ref{prop:degree}.
Since the degree is independent 
of the  choice of a smooth member $S$
of $|F|$ by Proposition \ref{prop:indep},
we obtain that $d(S')>0$ for any smooth member
$S'\in |F|$.
By Proposition \ref{prop:degree},
this implies $\kappa\inv(S')=2$ for any 
smooth $S'\in |F|$.

Once this is obtained, the equality $a(Z)=3$ follows
from a general estimate of the algebraic dimension from below 
in the case of a fibre space.
We refer to \cite[Proposition 4.1]{Fu04} and \cite[Corollary 4.3]{Kr99} for
details. 
\proofend

\vsp
In order to determine the algebraic dimension of $Z$ if
$\dim |F|=1$, we are left with the case where the pencil $|F|$ contains a 
smooth member $S$ satisfying $\kappa\inv(S)=1$.
This is well understood for the case $n=4$, which we next explain
in order to clarify the difference to the case $n>4$.

Let $Z$ be a twistor space on $4\CP^2$ satisfying $\dim|F|=1$
and let $C$ be the base locus of $|F|$.
Suppose that $Z$ has a smooth fundamental divisor of anti-Kodaira dimension
one. Using \eqref{ineq1}, this implies $a(Z)\le 2$. 
As we assume $n=4$, we can apply \cite[Theorem 6.2]{Kr98} to conclude that
$F$ is nef. 
Therefore, $K_{S}\inv\simeq F|_{S}$ is nef for each smooth $S\in|F|$,
i.e.\ $N=0$ in the Zariski decomposition of $C$ on $S$. 
As $K_{S}^{2}=0$, Corollary \ref{cor:finord} shows now that
$\kappa\inv(S)=1$ if and only if 
the line bundle $K_{S}\inv|_C$ is of finite order in $\Pic^0(C)$. Because
\begin{align}\label{n4iso}
K\inv_S\big|_C&\simeq (F|_S)\big|_C\simeq F|_C,
\end{align}
this order does not depend on $S$.
Hence the anti-Kodaira dimension of smooth members
of $|F|$ is constant.
From this, exactly like at the end of the proof of Proposition \ref{prop:ak2},
we obtain the following result. 

\begin{proposition}\label{prop:a2}
Let $Z$ be a twistor space over $4\CP^2$ satisfying
$\dim |F|=1$.
Suppose that there exists a smooth member $S$ 
of the pencil $|F|$ satisfying $\kappa\inv(S)=1$.
Then  $a(Z)=2$. 
\end{proposition}

We should mention that
a much more concrete result was obtained 
in \cite[Theorem 3.4]{CK99} without using
the general estimate on the algebraic dimension of a fibred space.
Namely, if $\tau$ denotes the order of 
the line bundle $K_S\inv|_C$ in $\Pic^0(C)$, 
then $\dim|\tau F| = \tau+1$, $\Bs\,|\tau F| = \emptyset$,
and the associated morphism $Z\to\CP^{\tau+1}$ provides
an algebraic reduction of $Z$ which is an elliptic fibration.
Strictly speaking, the paper \cite{CK99} assumes
that the base curve $C$ is smooth, but the proof equally works 
even if it is  a cycle of rational curves because $K_S\inv|_C$ is trivial on
each component of $C$ in the situation considered here.
Although we are assuming $\dim |F|=1$
here, the case $\tau=1$ actually happens,
and then $\dim |F|=2$, see \cite{HI00, Hon00}.

Now we are ready to state the main result of this article.
It says that, in contrast to the case $n=4$, 
when $n>4$ and $\kappa\inv(S)=1$ for some $S\in|F|$,
equality in \eqref{ineq1} never holds.

\begin{theorem}\label{thm:main}
Let $n>4$ and $Z$ be a twistor space on $n\CP^2$ 
satisfying $\dim |F|=1$.
Suppose there exists a smooth member $S$ 
of the pencil $|F|$ such that $\kappa\inv(S)=1$,
then  $a(Z)=1$. 
\end{theorem}

If such a member $S\in |F|$ exists,
inequality \eqref{ineq1} implies $a(Z)\le 2$.
On the other hand we have $a(Z)\ge 1$ from the 
presence of the pencil.
Thus for the proof of Theorem \ref{thm:main}
it is enough to show $a(Z)\neq 2$.
We emphasize here that in Theorem \ref{thm:main} we are {\em not} 
assuming that the member $S$ is 
generic in the pencil $|F|$.
Indeed, by \cite[Corollary 4.3]{Kr99}, under the assumption
of Theorem \ref{thm:main}, 
if $S$ is  generic in the pencil $|F|$,
equality in \eqref{ineq1} holds, which means
\begin{align}\label{}
a(Z) = 1 + \kappa\inv(S)=2.
\end{align}
Hence Theorem \ref{thm:main} implies that 
{\em the member $S\in |F|$ in the theorem
%(which satisfies $\kappa\inv(S)=1$)
cannot be generic in the pencil $|F|$}.
In other words, even if a twistor space $Z$ on $n\CP^2$,
$n>4$, with $\dim |F|=1$ possesses a smooth member $S\in|F|$
which satisfies $\kappa\inv(S)=1$, 
a {\em generic} member $S'$ of the pencil has to satisfy 
$\kappa\inv(S')=0$. 

We note that, by investigating small deformations of twistor spaces
of Joyce metrics \cite{J95}, the existence of twistor spaces that fulfil the
properties of Theorem \ref{thm:main} was shown in \cite{Fu04}. Our result
shows that they have algebraic dimension $1$.

The proof of Theorem \ref{thm:main} will be completed at the end of this
section. 
In preparation for this proof, we first note that under the assumptions of
Theorem \ref{thm:main} the base curve $C$ of $|F|$ cannot be smooth and
moreover $k\ge 2$ holds in 
\eqref{cycle1}.

\begin{proposition}\label{prop:cycle2}
Let $Z$ and $S$ be as in Theorem \ref{thm:main},
and $C$ be the unique anti-canonical curve on $S$.
Then  $C$  is a cycle of rational 
curves as in \eqref{cycle1} with $k\ge 2$.
\end{proposition}

\proof
By Proposition \ref{prop:cycle}, the base curve $C$ is either a smooth
elliptic curve 
or a cycle of rational curves as in \eqref{cycle1}.
If $C$ is smooth, from $K^2=8-2n<0$
we easily obtain $h^0(mK_S\inv)=1$ for any $m>0$, but then $\kappa\inv(S) = 0$.
Therefore $C$ cannot be smooth.

If $C$ is as in \eqref{cycle1} with $k=1$, then $C= C_1 + \ol C_1$, and
$C_1\ol C_1=2$ as $C$ is a cycle.
Together with $K^2 = 8-2n$, this means $C_1^2=\ol C_1^2=2-n$.
Hence the intersection matrix for the cycle $C$ becomes
$$
\begin{pmatrix}
2-n & 2 \\ 2 & 2-n
\end{pmatrix}.
$$
If $n>4$, this is negative definite. 
Hence $P=0$ in the Zariski decomposition $C=P+N$ of $C$,
which implies  $\kappa\inv(S)=0$, see Lemma \ref{lem:Pzero} (i).
Therefore $k\ge 2$.
\proofend

\vsp
Let $Z, S$ and $C$ be as in 
Theorem \ref{thm:main} and Proposition \ref{prop:cycle2}
and let $\mu:\hat Z\to Z$ be the blow-up with centre $C$.
The space $\hat Z$ has singularities, but we discuss this
later.
Let $E_i$ and $\ol E_i$ 
($1\le i\le k$) be the exceptional divisors
over the curves $C_i$ and $\ol C_i$, respectively.
Write 
$$
E:=(E_1+\dots+ E_k) +(\ol E_1+\dots + \ol E_k)
$$
for the sum of all exceptional divisors.
Any two distinct smooth members $S'$ and $S''$ of the pencil $|F|$ 
intersect transversally along the cycle $C$
in the sense that $S''|_{S'}=C$ as a divisor on $S'$.
Therefore, the blow-up $\mu:\hat Z\to Z$ composed with the rational
map $Z\to \CP^1$ associated to the pencil $|F|$
is a morphism $$\hat f:\hat Z\to \CP^1,$$
the fibres of which are the strict transforms of the members of 
the pencil $|F|$.
This is nothing but an elimination of the indeterminacy locus
of the rational map $Z\to \CP^1$.
We will use the same letters to denote divisors in $Z$ and their strict
transforms in $\hat Z$. 

Recall from Proposition \ref{prop:d1} that the pencil $|F|$ has 
exactly $k$ reducible members $S_i^++S_i^-$,
and that  $k\ge 2$ by Proposition \ref{prop:cycle2}.
Let $\lmd_i\in\CP^1$, $1\le i\le k$, be the point
over which the reducible fibre $S_i^+\cup S_i^-$ is lying.
Evidently we have the basic relation 
\begin{align}\label{basic1}
\hat f^* \ms O(1) \simeq \mu^* F - E.
\end{align}
For any index $i$, the exceptional divisor $E_i$ 
comes with two fibrations $\hat f|_{E_i}:E_i\to \CP^1$
and $\mu|_{E_i}: E_i\to C_i\simeq\CP^1$.
Both are clearly $\CP^1$-bundles, and fibres of different fibrations
intersect transversally at one point.
Hence we obtain an isomorphism $E_i\simeq\CP^1\times\CP^1$
for any $i$.
The same is true for the real conjugate divisor $\ol E_i$ from reality.

Since the centre $C$ of the blow-up $\mu$ is singular, being a cycle of
rational curves, the space $\hat Z$ has ordinary double points over 
the singularities of the cycle $C$.
Concretely, these are exactly the points where the four divisors
\begin{align}\label{odp1}
S_i^+ , S_i^-, E_i \,{\text{ and }} \, E_{i+1}
\quad (1\le i\le k)
\end{align}
meet and the real conjugate points of these,
see Figure \ref{fig1} where these points for $i=1$ are indicated
by black circles.
Thus $\hat Z$ has precisely $2k$ ordinary double points in total.
Since each ordinary double point admits two small resolutions,
the number of simultaneous small resolutions
of all the double points, which are compatible with
the real structure, is $2^k$ in total.
Let $Z_1\to \hat Z$ be any one of these small resolutions.
We denote by $\mu_1:Z_1\to Z$ the composition
$Z_1\to \hat Z\to Z$ of 
 the small resolution and  the blow-up.
We continue to use the same letters $E_i$ and $\ol E_i$ to denote
the strict transforms of the exceptional divisors in $Z_1$.
Also, $E$ denotes the total sum $(E_1 + \dots + E_k)
+ (\ol E_1+ \dots + \ol E_k)$ in $Z_1$.
We write $f_1:Z_1\to\CP^1$ for the composition
$Z_1\to \hat Z\to\CP^1$.
From \eqref{basic1} we obtain the relation 
\begin{align}\label{basic2}
 f_1^* \ms O(1) \simeq \mu_1^* F - E,
\end{align}
and $f_1$ has reducible fibre $S_i^+\cup S_i^-$ over
the point $\lmd_i$, $1\le i\le k$.

If we regard the cycle $C$ as a curve on a real smooth member of the pencil
$|F|$, it has a Zariski decomposition with $P$ and $N$ both real
\begin{align}\label{Z22}
C = P + N.
\end{align}
By Proposition \ref{prop:indep}, this Zariski decomposition is independent of 
the choice of the smooth member.
As in Section \ref{s:Z}, let $m_0$ be the smallest integer for which $m_0P$ is
integral, and write
$$
m_0P = l_1 C_1 + \dots + l_k C_k + l_1 \ol C_1 + \dots + l_k\ol C_k.
$$
For later use, we use these coefficients $l_{i}$ and define a $\QQ$-divisor
$\bm P$ on $Z_1$ by 
the equation
\begin{align}
m_0 \bm P = l_1 E_1 + \dots + l_k E_k + l_1 \ol E_1 + \dots + l_k\ol E_k.
\end{align}
Similarly we define a $\QQ$-divisor $\bm N$ on $Z_1$ by the equation
\begin{align}\label{N}
\bm P + \bm N =  E.
\end{align}
Thus $\bm P$ and $\bm N$ may be regarded as enlargements of 
the nef part and the negative part, respectively, 
of the Zariski decomposition \eqref{Z22} to divisors in $Z_1$. 
The key feature of these divisors is that if we restrict \eqref{N} to a smooth
fibre of the projection $\mu_1:Z_1\to\CP^1$,
we obtain the Zariski decomposition \eqref{Z22}.

Because we assume $\kappa^{-1}(S)=1$ for a given $S\in|F|$, we obtain from
Proposition \ref{prop:degree} and Lemma \ref{lem:Pzero} (i) that $P^{2}=0$ and
$P\ne 0$ in $S$. Because of Proposition \ref{prop:indep} the same is true for
all smooth fundamental divisors. 
In particular, we can apply
Proposition \ref{prop:Z2} to get $l_i>0$ for all $i$ and $l_{i}=1$ for 
at least one $i$.
Moreover, by Proposition \ref{prop:Z2} (iv),
as $n>4$, the $k$ integers $l_i$ cannot all be equal to each other.
Continuing to write $l_{k+1}:= l_1$, it follows that there is an
index $i$ for which $l_i>l_{i+1}$ holds. 
After a cyclic permutation of the indices we may assume that $l_1>l_2$.

For arbitrary integers $r,\rho$ we define a line bundle $\ms M(r,\rho)$
on $Z_1$ by
\begin{align}\label{M1}
\ms M(r,\rho) := f_1^*\ms O(r) + \rho m_0\bm P. 
\end{align}

\begin{proposition}\label{prop:rest1}
Let $S'$ be any smooth fibre of the fibration $f_1:Z_1\to\CP^1$,
and identify the pair $(S',S'\cap E)$ in $Z_{1}$ with
$(S',C)$ in $Z$ by the map $\mu_1:Z_1\to Z$.
Then for any $r\in \ZZ$ and $\rho\in\ZZ$, we have
\begin{align}\label{rest3}
\ms M(r,\rho)
\big|_{S'\cap E}
\in \Pic^0(C).
\end{align}
\end{proposition}

\proof
Recall that $\Pic^{0}(C)$ denotes the group of those line bundles on $C$ that
are of degree zero on each component of $C$.
By the identifications $E_i\cap S'\simeq C_i$ and 
$\ol E_i\cap S' \simeq \ol C_i$
induced by the birational morphism $\mu_1$, we obtain that 
the restriction of $\ms M(r,\rho)$ to $S'$ is isomorphic to 
$
\rho m_0 P.
$
Here we have disposed of $f_1^*\ms O(r)$ since $S'$ is a fibre of $f_1$.
Because $P^{2}=0$ on $S'$, Lemma \ref{lem:Pzero} (ii) implies that  
$m_0P|_{S'\cap E}\in \Pic^0(C)$, hence so is $\rho m_0P|_{S'\cap E}$.
\proofend

\vsp

We now take a closer look at the ordinary double point
$S_1^+\cap S_1^-\cap E_1\cap E_2$ in $\hat Z$.
Without loss of generality, we may suppose that
$S_1^+$ and $S_1^-$ are distinguished 
by the property that the intersection  
$S_1^+\cap E_1$ is (not a curve but) a point, see Figure \ref{fig1}.

\begin{figure}
\includegraphics{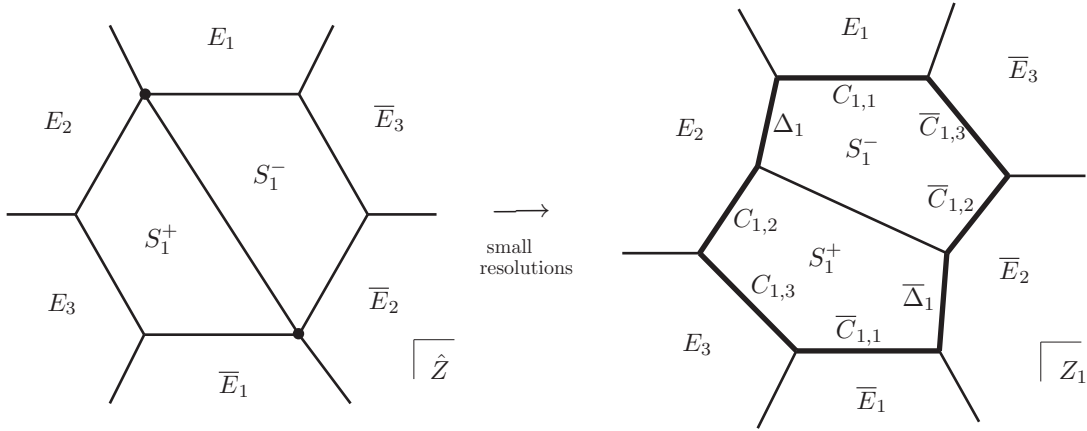}
\caption{
The small resolution of the ordinary double points on 
the fibre $\hat f\inv(\lmd_1)$ depicted for $k=3$.
The bold lines in the right figure mark the cycle
$\ms C$ that is defined below.
}
\label{fig1}
\end{figure}

The small resolution $Z_1\to \hat Z$ blows up one of
the pairs $\{S_1^-,E_2\}$ or $\{S_1^+, E_1\}$.
Because there is no essential difference between these two,
we may suppose that the pair $\{S_1^-,E_2\}$ is blown up.
We denote the exceptional curve over the point $S_1^+\cap E_1$ by $\Delta_1$. 

Since $\mu$ blows up the cycle $C$,
we cannot speak about strict transforms of the components $C_i$ of $C$
into $Z_1$. 
But, if we specify a fibre of $f_1:Z_1\to\CP^1$, we can.
In particular, we denote the strict transforms of $C_{i}$ and $\ol C_{i}$
in the fibre $f_1\inv(\lmd_1)$ by
\begin{align}\label{c1i}
C_{1,i}
\quad{\text{and}}\quad
\ol C_{1,i}, \quad 1\le i\le k.
\end{align}
These are identified with $C_i$ and $\ol C_i$, respectively,
by the birational morphism $\mu_1:Z_1\to Z$.
The first  index, $1$, indicates that the curves are included in 
the fibre $f_1\inv(\lmd_1)=S_1^+\cup S_1^-$, and 
the second  index reflects that they correspond to
$C_i$ and $\ol C_i$ respectively.
Unlike in the original space $Z$, the union of all the curves 
\eqref{c1i} is not a cycle 
because the exceptional curves  $\Delta_1$ and $\ol{\Delta}_1$
are inserted between $C_{1,1}$ and $C_{1,2}$, and between
$\ol C_{1,1}$ and $\ol C_{1,2}$, respectively, see Figure \ref{fig1}.
Hence the union of all the curves \eqref{c1i}
consists of two connected chains
\begin{align}\label{chain1}
C_{1,2}\cup C_{1,3}\cup\dots\cup C_{1,k}\cup \ol C_{1,1}
\quad{\text{and}}\quad
\ol C_{1,2}\cup \ol C_{1,3}\cup\dots\cup \ol C_{1,k}\cup  C_{1,1},
\end{align}
and by adding $\Delta_1$ and $\ol{\Delta}_1$ we get
a cycle of rational curves, which we denote by $\ms C$.
We clearly have
\[
\ms C=f_1\inv(\lmd_1)\cap E
\]
in the sense that $\ms C$ is the restriction of the divisor $E$ to the fibre 
$f_1\inv(\lmd_1)$ or, equivalently, $\ms C$ is an element of the linear system
$|(f_1|_{E})^{\ast}\ms O(1)|$ on the surface $E$.

\vsp
The intersection numbers of the line bundle $\ms M(r,\rho)$ with the 
components of the cycle $\ms C$ will play an essential role in the 
subsequent proofs.

\begin{lemma}\label{lemma:int1}
We have
\begin{align}\label{int100}
\ms M(r,\rho)\cdot C_{1,i} =
\begin{cases}
\hfil 0 & i\neq 2\\
-\rho (l_1-l_2) & i=2.
\end{cases}
\end{align}
and 
\begin{align}\label{int101}
\ms M(r,\rho)\cdot \Delta_1 = \rho(l_1-l_2).
\end{align}
\end{lemma}

\proof
Let $S'$ be any smooth fibre of $f_1:Z_1\to\CP^1$.
If $i\neq 2$, the curve $C_{1,i}$ is homologous in $E_i$ to $C_i=S'\cap E_i$.
Therefore, using Proposition \ref{prop:rest1}, we have
\begin{align}\label{int01}
\ms M(r,\rho)\cdot C_{1,i}=0
\quad\text{ for all } i\neq 2.
\end{align}
The case $i=2$ requires more attention. 
Dropping components that are disjoint from $C_{1,2}$,
see Figure \ref{fig1}, we first obtain
\begin{align}\label{int07}
\ms M(r,\rho)
\cdot C_{1,2} = \rho(l_2E_2+l_3E_3) C_{1,2}.
\end{align}
Further, as $C_{1,2}\cap E_3$ is one point and 
the intersection is transverse,
we have $ E_3 C_{1,2}= 1$.
Next,  $E_2 C_{1,2}= (C_{1,2})^2_{S_1^+}$ since $E_2\cap S_1^+
= C_{1,2}$ and the intersection is transverse.
Moreover, since the pairs $(S_1^+, C_{1,2})$
in $Z_1$ and $(S_1^+,C_1)$ in the original twistor space $Z$
are isomorphic, the self-intersection number
$C_{1,2}^2$ in $S_1^+ \subset Z_{1}$ is equal to $C_2^2$ calculated in
$S_1^+ \subset Z$.
We let $a_2:=-C_2^2$, calculated in $S\subset Z$.
By the adjunction formula and because
$S_1^+ + S_1^- = F$ on $Z$, we have 
\begin{align*}
  (C_2)^2_{S_1^+} &= -2 + K\inv_{S_1^+}C_2
  = -2 + (K_Z\inv - S_1^+) C_2
  = -2+ (F + S_1^-)C_2 \\
  &= -2 + FC_2 +S_1^-C_2
  = -2 + K\inv_{S}C_2 + 1
  = (C_2)^2_S+1 =
  -a_2+1.
\end{align*}
By Lemma \ref{lem:Pzero} (ii) we have $m_0P C_2 = 0$ on the original
surface $S\subset Z$ of which we assumed $\kappa^{-1}(S)=1$.
Thus, we have the relation
\begin{align}\label{rel1}
l_1 -a_2l_2 + l_3 = 0,
\end{align} 
with which we obtain
\[
 (l_2E_2+l_3E_3)
 C_{1,2}\notag
  = l_2(-a_2+1) + l_3=l_2- l_1.
\]
Therefore, from \eqref{int07}, we get
\begin{align}\label{int02}
\ms M(r,\rho) \cdot C_{1,2} = -\rho (l_1-l_2).
\end{align}

For the remaining intersection number
\eqref{int101}
we notice that the curve $\Delta_1 \cup C_{1,2}$, regarded
as a curve in the surface $E_2$, is homologous to 
the fibre $S\cap E_{2} \simeq C_{2}$ 
of the projection $f_1|_{E_2}:E_2\to\CP^1$.
Because $\ms M(r,\rho)|_{S} = m_{0}P$
and $\ms M(r,\rho)\cdot (S\cdot E_{2})
= (\ms M(r,\rho)\cdot S)\cdot E_{2} = (m_{0}PC_{2})_{S}$ is zero,
we obtain 
$$
\ms M(r,\rho)\cdot \big(\Delta_1 + C_{1,2}\big) =0,
$$
and \eqref{int101} follows from \eqref{int02}.
\proofend

\vsp
Let $\rho>0$ be an integer and
recall that we have chosen notation so that $l_{1}>l_{2}$.
From \eqref{int100} for $i=2$
we then obtain $\ms M(r,\rho)\cdot C_{1,2}<0$. 
Using \eqref{int100} for $i\ne 2$ and reality it follows now that the two
chains 
\begin{align}\label{bc1}
C_{1,2}\cup C_{1,3}\cup\dots\cup C_{1,k}\cup \ol C_{1,1}
\quad{\text{and}}\quad
\ol C_{1,2}\cup \ol C_{1,3}\cup\dots\cup \ol C_{1,k}\cup  C_{1,1}
\end{align}
are contained in the base locus of the linear system $|\ms M(r,\rho)|$.
This will be a stepping stone for the following stronger statement, which will
play a key role in the proof of the main theorem.

\begin{proposition}\label{prop:keyvan}
For any $r\in\ZZ$ and $\rho>0$, we have
$$
H^0\big(E, \ms M(r,\rho)|_E\big) = 0.
$$
In particular, if $\rho>0$ the divisor $E$ is a fixed component of the linear
system $|\ms M(r,\rho)|$ on $Z_{1}$.
\end{proposition}

\proof
Recall that the cycle of rational curves $\ms C$  
belongs to the linear system $|(f_{1}|_E)^{\ast}\ms O(1)|$ on $E$ and
that $\ms M(r,\rho) = \ms M(r-1,\rho)\otimes f_{1}^{*}\ms O(1)$. Therefore,
multiplication by a section of $f_{1}^{*}\ms O(1)|_E$
with zero locus $\ms C$ provides an exact sequence 
\[
0 \lras \ms M(r-1,\rho)|_E \lras \ms M(r,\rho)|_E 
\lras \ms M(r,\rho)|_{\ms C} \lras 0.
\]
To first show that $\phi = 0$ in the induced exact sequence
\begin{equation}
  \label{eq:phi-is-zero}
  0 
  \lras H^{0}\left(\ms M(r-1,\rho)|_E\right)
  \lras H^{0}\left(\ms M(r,\rho)|_E\right)
  \stackrel{\phi}{\lras} H^{0}\left(\ms M(r,\rho)|_{\ms C}\right)\,,
\end{equation}
we let 
$$
s\in H^0\big(E,\ms M(r,\rho)|_E\big)
$$
be a non-zero section and note that $\phi(s) = s|_{\ms C}$.

Suppose that $s$ identically vanishes on a component of $E$.
Then since, for each $i$ the degree of the line bundle $\ms M(r,\rho)$ on a
general fibre of $f_1|_{E_i}:E_i\to\CP^1$  is zero by 
Proposition \ref{prop:rest1},
we obtain that $s$ vanishes identically on $E$.
Hence $s$ cannot vanish identically on any component
of $E$.
Therefore we can consider the effective Cartier divisor $(s)$ on $E$, and the
restriction $(s|_{E_i})$ is a curve on $E_i$ for any $i$.

From the intersection numbers \eqref{int100} we have seen above that
the divisor $(s)$ includes the two connected chains \eqref{bc1} for
$\rho>0$ and all $r\in\ZZ$.
Hence the curve $(s|_{E_2})$ passes through
the point $C_{1,1}\cap \Delta_1$, see Figure \ref{fig1}.
This point is not on $C_{1,2}$, because $C_{1,1}$ and $C_{1,2}$ are disjoint
curves in $f_1\inv(\lmd_1)\subset Z_{1}$. 
If the divisor $(s|_{E_2})$ had a component other than  $\Delta_1$, containing
this point, then we would have $s|_{E_2}= 0$ by Proposition \ref{prop:rest1}, 
since the fibre of $f_1|_{E_2}$ through
this point is $\Delta_1+C_{1,2}$ and so 
$(s|_{E_2})$ would have to intersect any fibre of the projection
$f_1:E_2\to\CP^1$. 
Therefore such a component does not exist.
Hence the divisor $(s|_{E_2})$, and so $(s)$, includes $\Delta_1$.
The same argument shows that $(s)\supset\ol{\Delta}_1$.
We have seen before that $(s)$ includes all the other components
$C_{1,i}$ and $\ol C_{1,i}$ of $\ms C=f_1\inv(\lmd_1)\cap E$ and so
we obtain $\ms C\subset (s)$, which implies $\phi(s) = s|_{\ms C} = 0$.
This shows that $\phi=0$. 

From the exact sequence \eqref{eq:phi-is-zero} we obtain now for all $\rho>0$
and all $r\in\ZZ$ that 
\begin{equation}
  \label{eq:isomorphic}
  H^0\big(E, \ms M(r-1,\rho)|_E\big) \cong H^0\big(E, \ms M(r,\rho)|_E\big)\,.
\end{equation}
On the other hand, using \eqref{M1} and the 
projection formula for $f_1:E\to\CP^1$, we obtain
\begin{align*}
  f_{1*} \ms M(r,\rho)|_E
  &\cong f_{1*} \bigl(f_1^*\ms O_{\CP^1}(r)
    \otimes \ms O_{E}(\rho m_0\bm P)\bigr)\\
  &\cong \ms O_{\CP^1}(r) \otimes f_{1*} \ms O_{E}(\rho m_0\bm P)\,.
\end{align*}
For any coherent sheaf $\ms F$ on $\CP^{1}$ and for sufficiently large
$j$, $H^{0}\left(\ms O_{\CP^1}(-j) \otimes \ms F\right) = 0$.
Taking $j$ so that this holds for 
$\ms F= f_{1*} \ms O_{E}(\rho m_0\bm P)$,
 we obtain from \eqref{eq:isomorphic} that
\begin{align*}
H^0\big(E, \ms M(r,\rho)|_E\big)&\cong
H^0\big(E, \ms M(-j,\rho)|_E\big)\\
&\cong H^0\big(\CP^{1}, f_{1*}\ms M(-j,\rho)|_E\big) \\
&\cong H^0(\CP^1, \ms O_{\CP^1}(-j)\otimes \ms F) = 0
\end{align*} for all $r\in\ZZ$ and
all $\rho>0$.
\proofend

\vsp
By Propositions \ref{prop:rest1} the restriction
of $m_{0}\bm P$ to $f^{-1}(\lambda)\cap E\simeq C$ 
gives an element $\ms P_{\lambda}\in\Pic^{0}(C)$ for each value of $\lambda$
for which the fibre $f_{1}^{-1}(\lambda)$ is smooth.

\begin{proposition}\label{prop:Lnu2}
  Assume that $\ms P_{\lambda}$ has the same finite order $\tau$ for all
  smooth fibres $f_{1}^{-1}(\lambda)$.  
  Then, for any $r\in\ZZ$ and $\nu>0$, we have 
  \[
  \big|\ms M(r,\nu\tau)\big|
  = \big| f_1^*\ms O(r)\big| + 
  \nu\tau m_0\bm P.
  \]
\end{proposition}

\proof
Recall that 
$$
\ms M(r,\nu\tau) = f_1^*\ms O(r) + \nu\tau m_0\bm P.
$$
By Proposition \ref{prop:keyvan} the divisor $E$ is a fixed component of
$\big|\ms M(r,\nu\tau)\big|$ and so
\begin{align}\label{mrnu}
\big|\ms M(r,\nu\tau)\big|= \big| f_1^*\ms O(r) + 
\nu\tau m_0\bm P-E\big| + E.
\end{align}
Let $S'$ be a general fibre of  $f_1:Z_1\to\CP^1$,
and identify the pair $(S',S'\cap E)$ with
the pair $(S',C)$ included in the original space $Z$
by the birational morphism $\mu_1:Z_1\to Z$.
Because we assumed $\ms P_{\lambda}$ to have finite order, we can use
Proposition \ref{prop:Z3} (iv) to obtain 
\begin{align}
\big|
\nu\tau m_0P-C
\big|
=\big|(\nu-1)\tau m_0P\big| +
( \tau m_0P - C)
\end{align}
on $S'$. As this is valid for any smooth fibre of $f_1$, we obtain that 
the divisor $\tau m_0\bm P - E$ on $Z_1$ is 
a fixed component of the system 
$| f_1^*\ms O(r) + 
\nu\tau m_0\bm P-E|$.
Hence, from \eqref{mrnu} and using definition \eqref{M1}, we obtain 
\begin{align*}%\label{mrnu2}
\big|\ms M(r,\nu\tau)\big|&= \big| f_1^*\ms O(r) + 
(\nu-1)\tau m_0\bm P\big| + \tau m_0\bm P\\
&= \big| \ms M(r,(\nu-1)\tau) \big| + \tau m_0\bm P.
\end{align*}
Using induction on $\nu>0$ we finally obtain 
\[
\big|\ms M(r,\nu\tau)\big|
=\big|\ms M(r,0)\big|+ \nu \tau m_0\bm P
= \big| f_1^*\ms O(r)\big| +\nu \tau m_0\bm P,
\]
as required.
\proofend

\vsp
We are now ready to prove our main result.

\vsp
\noindent
\emph{Proof of Theorem \ref{thm:main}.}
We first show that if the order of the restriction $\ms P_{\lambda}$ of
$m_{0}\bm P$ to $C\simeq f^{-1}(\lambda)\cap E$ is not the same for all smooth
fibres $f_{1}^{-1}(\lambda)$ then we have $a(Z)=1$.
To see this, we note that the set of points of finite order in
$\Pic^{0}(C)$, 
which is isomorphic to $\mathbb{C}^{\ast}$, has the property that any
continuous path that connects two distinct points  of finite order
necessarily contains points of infinite order. Moreover 
we have a continuous function
$\lambda\mapsto \ms P_{\lambda}\in\Pic^{0}(C)$ which is defined on the open
set of $\CP^{1}$ over which $f_{1}$ has smooth fibres. 
Therefore, if the order of $\ms P_{\lambda}$ is not constant, the image of the
continuous function mentioned above contains an element of infinite order. 
If $S'\in |F|$ is the member corresponding to such an element of infinite
order, Corollary \ref{cor:finord} shows that $\kappa\inv(S')\neq 1$.
As $d(S')=d(S)= 0$, Proposition \ref{prop:degree} implies 
$\kappa\inv(S')\neq 2$, and we conclude $\kappa\inv(S')=0$. 
As we have seen from \eqref{ineq1}, this implies $a(Z)=1$.

Thus it remains to consider the situation in which the order of 
$\ms P_{\lambda}$ is the same for all smooth fibres $f_{1}^{-1}(\lambda)$.
We will show by contradiction that this situation never happens.

So suppose that the restriction $\ms P_{\lambda}$ has the same finite order
$\tau$ for all $\lambda$ for which $f_{1}^{-1}(\lambda)$ is smooth. Then, by
Corollary \ref{cor:finord}, we have $\kappa^{-1}(S)=1$ for each smooth
$S\in|F|$. Hence $\kappa\inv=1$ for generic members of the pencil 
$|F|$ and, using \cite[Corollary 4.3]{Kr99}, we obtain that $a(Z)=2$.

On the other hand, with the aid of Proposition \ref{prop:Lnu2}, we are able to
show $a(Z)=1$ as follows.
For each integer $\nu>0$, \eqref{Z22} gives a Zariski decomposition in $S$ 
\[
\nu\tau m_0 C = \nu\tau m_0P + \nu\tau m_0N,
\]
and since the Zariski decomposition of the cycle $C$ is 
independent of the smooth member $S$ by Proposition \ref{prop:indep},
this is the Zariski decomposition of the divisor 
$\nu\tau m_0C$ on any smooth member of the pencil $|F|$.
Therefore, by Proposition \ref{prop:Z1},
for any non-singular fundamental divisor $S'\in |F|$,
the divisor $\nu\tau m_0N$ 
is a fixed component of the system $|\nu\tau m_0K\inv|$ on $S'$.
Hence the pull-back $\mu_1^*(\nu\tau m_0F)$ to $Z_1$ has the divisor 
$
\nu\tau m_0\bm N
$
as a fixed component; the divisor $\bm N$ was defined in \eqref{N}.
We define a line bundle $\ms L$ on $Z_1$ by
\begin{equation}\label{L1r}
\ms L := 
\mu_1^*(\tau m_0F) - \tau m_0\bm N
\simeq f_1^*\ms O(\tau m_0) + \tau m_0\bm P
=\ms M(\tau m_{0},\tau)\,,
\end{equation}
where we have used the relation \eqref{basic2} to get the isomorphism.
The line bundle $\ms L$ is equipped with a natural real structure
coming from that on $F$ on $Z$.
The morphism $\mu_{1}:Z_{1}\lra Z$ provides a bijection between the members of
the two linear systems $|\nu\tau m_0F|$ on $Z$ and
$|\ms L^{\otimes\nu}|$ on $Z_1$.
Moreover, the composition of $\mu_{1}$ with the rational map defined by
$|\nu\tau m_0F|$ is the rational map which is given by
$|\ms L^{\otimes\nu}|$. 
Because 
\[
|\ms L^{\otimes\nu}| = |\ms M(\nu\tau m_{0},\nu\tau)|
= \left|f_{1}^{\ast}\ms O(\nu\tau m_{0})\right| + \nu\tau m_0\bm P,
\]
by Proposition \ref{prop:Lnu2}, we see that the rational map
associated to $|\ms L^{\otimes\nu}|$ has one-dimensional image, 
equal to the image of a Veronese embedding of $\CP^{1}$.
Hence, for all $\nu>0$, the rational map defined by $|\nu\tau m_0F|$ has
one-dimensional image as well, and this implies that $a(Z)=\kappa(Z,F)=1$.
This contradicts our previous conclusion $a(Z)=2$.
Hence it is impossible that $\ms P_{\lambda}$ has finite order not depending
on $\lambda$.
\proofend

\medskip
It might be interesting to observe that in the argument that leads to the
contradiction in the second part of the above proof, the constancy of the
order of $\ms P_{\lambda}$ is only needed to apply Proposition \ref{prop:Lnu2}.

\section{Comment on a result in the paper \cite{Fu04}}
In \cite{Fu04}, the existence of a twistor space on $n\CP^2$ with algebraic
dimension two is claimed  for any $n>4$. 
The proof of this assertion consists essentially of the following two parts.
\begin{enumerate}\renewcommand{\labelenumi}{\arabic{enumi}.}
\item
For a non-singular surface $S$ which admits an anti-canonical cycle,
it is shown in \cite[Lemma 3.2]{Fu04} that the finiteness of the order of a
certain line bundle $L_a$ on a cycle of rational curves implies
$\kappa\inv(S)=1$. With the aid of formula \eqref{ZD1} it can be shown that
the pull-back of the line bundle $L_a$ to  $S$ is isomorphic to the line
bundle $\ms O(m_0P)|_C$ studied in Section 2.
\item
If a twistor space $Z$ on $n\CP^2$,
$n>4$, admits a divisor in $|F|$ which is biholomorphic to the surface $S$ in
the above part 1 satisfying $\kappa\inv(S)=1$,
then we always have $\dim |F|=1$.
It is claimed in the proof of \cite[Lemma 4.3]{Fu04} that the isomorphism class
of the line bundle $L_a$ from part 1 is the same for generic members of
the pencil $|F|$. 
\end{enumerate}
It follows that a generic fundamental divisor $S$ in a twistor space $Z$ as in
part 2 above satisfies $\kappa\inv(S)=1$. Using \cite[Proposition 4.1]{Fu04}
it is then concluded that $a(Z)=2$. 
This conclusion is in conflict with our Theorem \ref{thm}.

To resolve this contradiction, we first recall some notation of \cite{Fu04} and
then focus on a key step in the proof that the order of the line bundle $L_a$
does not depend on $S$.

Each non-singular fundamental divisor $S_t$ ($t\in\CP^1$) is given as an
iteration of blow-ups
\begin{align}\label{bseq}
S_t \stackrel{w_t}\lra T'\stackrel{u}\lra T
\stackrel{\aaa}\lra \CP^1\times\CP^1.
\end{align}
Here, $\aaa$ is the blow-up of $4$ points, all of which are
nodes of an anti-canonical cycle on $\CP^1\times\CP^1$ consisting of $4$
components. The morphism $u$ is a composition of blow-ups of which each 
centre is a node of the anti-canonical cycle.
Both $T'$ and $T$ are toric and their isomorphism classes are independent of
$t\in\CP^1$. 
On the other hand, $w_t$ blows up $4$ points, where each centre belongs
to the strict transforms of the $4$ exceptional curves of $\aaa$ which are
disjoint to each other. 
Thus the variation of the complex structure on $S_t$ is entirely encoded in
the variation of the $4$ points blown up under $w_{t}$. 
It is important to keep in mind that the sequence \eqref{bseq} is
constructed for each non-singular $S_t$ individually, and there is no natural
way to identify the surface $T'$ (and $T$ and $\CP^1\times\CP^1$ also) for 
different choices of $t\in\CP^1$.

Let $C$ be the anti-canonical cycle on $S_t$,
which is the base curve of the pencil $|F|$,
and $B'=w_t(C)$. 
The argument that the order of the line bundle $L_{a}$ does not
depend on the surface $S$ in the fundamental pencil is built in an essential
way on the assertion that the pull-back $(w_t|_{C})^*(K\inv_{T'}|_{B'})$ does
not depend on $t$. This independence is based on the claim that the morphism 
$w_t|_{C}:C\rightarrow B'$ does not depend on $t$.

We now derive a contradiction from the assumption that 
$w_t|_{C}:C\rightarrow B'$ is the same for two nearby values of $t$ for which
$S_{t}$ is a real smooth member of the fundamental pencil. This confirms that
there is a gap in the proof of \cite[Lemma 4.3]{Fu04} on line 10, page 1105. 

Let $S$ and $S'$ be two real irreducible fundamental divisors on a twistor
space $Z$ over $n\CP^{2}$ such that both are obtained be a sequence of
blow-ups from $\CP^{1}\times\CP^{1}$ as in \eqref{bseq}. 
The linear system of twistor lines on each of these smooth surfaces
gives, see \cite{PP94}, a morphism with target
$\CP^{1}$ and we can choose the morphisms in \eqref{bseq} so that their
composition with the first projection $p_{1}:\CP^{1}\times\CP^{1} \rightarrow
\CP^{1}$ is the morphism given by the pencil of twistor lines.
We now identify the surfaces $T'$, $T$ and $\CP^{1}\times\CP^{1}$ obtained by
blowing down some curves on $S$ with those obtained by blowing down some
curves on $S'$. As mentioned above, there is no natural way to do so. 

From Proposition \ref{prop:d1} we can deduce that the two component of the
cycle $C$ that are met by real twistor lines in a real fundamental divisor
$S_{t}$ do not change as long as $t$ varies within one connected component of
the set of real irreducible fundamental divisors, which is the complement of a
finite set of points on a circle. We assume that $S$ and $S'$ belong to such a
connected component and denote the morphisms $w_{t}$ for these two surfaces by
$w:S\rightarrow T'$ and $w':S'\rightarrow T'$, respectively. 
We let $\pi$ be the composition 
\[
\pi:S\stackrel{w}\lra T'\stackrel{u}\lra 
T \stackrel{\aaa}\lra \CP^1\times\CP^1\stackrel{p_{1}}\lra \CP^1
\]
and define $\pi':S'\lra\CP^{1}$ similarly. Both restrictions $w|_{C}$ and
$w'|_{C}$ have the same image $B'\subset T'$. Because 
$p_{1}\alpha u:T'\lra\CP^{1}$ is the same for both surfaces $S$ and $S'$, the
assumption that $w|_{C}=w'|_{C}$ implies that the two compositions 
\[
\pi|_C:C\stackrel{w|_C}\lra B'\subset T'\stackrel{u}\lra 
T \stackrel{\aaa}\lra \CP^1\times\CP^1\stackrel{p_{1}}\lra \CP^1
\]
and
\[
\pi'|_C:C\stackrel{w'|_C}\lra B'\subset T'\stackrel{u}\lra 
T \stackrel{\aaa}\lra \CP^1\times\CP^1\stackrel{p_{1}}\lra \CP^1
\]
coincide. Let now $L\subset S$ be a real twistor fibre, then $L\cap C =
\{p,\overline{p}\}$ is a set of two conjugate points. Consider the fibre
$L'\subset S'$ of $\pi'$ over the point $\pi(p)=\pi(\overline{p})$. As $\pi'$
has only finitely many reducible fibres and the real twistor lines in $S$ and
$S'$ meet the same components of $C$, we can choose $L$ so that $L'$ is
irreducible.  
Because, by assumption, $\pi|_C = \pi'|_C$, we have $L'\cap C =
\{p,\overline{p}\}$ as well. However, as there do not exist two real twistor
fibres in $Z$ through the point $p$ and $S\cap S'$ does not contain a real
twistor fibre, $L'$ cannot be a real member of the
pencil $|L'|$ in $S'$. Hence, $L'$ and its conjugate are two different members
of the same pencil on $S'$ and they intersects at least at the two points $p$
and $\overline{p}$. But this is absurd as the irreducible curve $L'$ has
self-intersection number $0$ on $S'$. This contradiction shows that
$w_t|_{C}:C\rightarrow B'$ cannot be the same for any two values of $t$ for
which $S_{t}$ is real. 

A detailed analysis of the argument in \cite[p.\ 1105]{Fu04} reveals that 
the core problem is whether for $s\ne t$ the automorphism
$(w_s|_C)\circ (w_t|_C)\inv:B'\to B'$ extends to an automorphism
$T'\to T'$, or more precisely, to an isomorphism between 
neighbourhoods of $B'$. This seems to have been overlooked in \cite{Fu04}.
A priori there is no reason for this to hold,
and our Theorem \ref{thm} shows that 
for generic $s\in \CP^1$, such an extension cannot exist and
even if some member in the pencil $|F|$ satisfies 
$\kappa\inv=1$, the general member of the pencil needs to satisfy
$\kappa\inv =0$.

\vspace{5mm}
\noindent
{\small
Nobuhiro Honda\\
Mathematical Institute, Tohoku University, and\\
Department of Mathematics, Tokyo Institute of Technology, Tokyo, JAPAN
(current affiliation)\\
{\tt{honda@math.titech.ac.jp}}

\bigskip
\noindent
Bernd Kreu\ss ler\\
Department of Mathematics and Computer Studies,
Mary Immaculate College,
Limerick, IRELAND\\
{\tt{bernd.kreussler@mic.ul.ie}}
}
\end{document}